\documentclass{amsart}[12pt]
\usepackage{mathrsfs, amsmath,amssymb}
\usepackage{fullpage}
\usepackage{graphicx}
\newtheorem{teo}{Theorem}
\newtheorem{pro}{Proposition}

\newtheorem{cor}{Corollary}
\newtheorem*{remark}{Remark}

\title{Generalized Stochastic areas and windings arising from Anti-de Sitter and Hopf fibrations}

\author[N. Demni]{Nizar Demni}
\address{IRMAR, Universit\'e de Rennes 1\\ Campus de
Beaulieu\\ 35042 Rennes cedex\\ France}
\email{nizar.demni@univ-rennes1.fr}

\keywords{Anti-de Sitter fibration; Hopf fibration; complex Hyperbolic ball; complex projectif space, Real hyperbolic space; Subelliptic heat kernel; Generalized Maass Laplacian.}
\subjclass[2010]{60J60; 53C17}
\begin{document}
\maketitle

\begin{abstract}
In the first part of this paper, we derive explicit expressions of the semi-group densities of generalized stochastic areas arising from the Anti-de Sitter and the Hopf fibrations. Motivated by the number-theoretical connection between the Heisenberg group and Dirichlet series, we express the Mellin transform of the generalized stochastic area corresponding to the one-dimensional Anti de Sitter fibration as a series of Riemann Zeta function evaluated at integers. In the second part of the paper, we derive fixed-time marginal densities of winding processes around the origin in the Poincar\'e disc and in the complex projective line. 
\end{abstract}

\section{Motivation: The Heisenberg group case}
The L\'evy stochastic area
\begin{equation*}
A_t:= \int_0^t B_s^1dB_s^2 - B_s^2dB_s^1, \quad t \geq 0, 
\end{equation*}
where $B := (B^1, B^2)$ is a planar Brownian motion, is a very interesting object in both probability theory and mathematical physics (\cite{Gav}). It arises naturally from the heat kernel of the three-dimensional Heisenberg group 
$H_3 = \mathbb{C} \times \mathbb{R}$ since the latter is endowed with the standard contact form written in local coordinates $(x,y,t)$: 
\begin{equation*}
\eta_H := dt + (xdy - ydx). 
\end{equation*}
This form is actually the pull-back of the standard K\"ahler form 
\begin{equation*}
\alpha_H := xdy - ydx 
\end{equation*}
on $\mathbb{C}$ with respect to the fibration: 
\begin{equation*}
\pi: H_3 \rightarrow \mathbb{C}, \quad (x,y,t) \mapsto (x,y). 
\end{equation*}
The L\'evy stochastic area has also a beautiful connection to the Dirichlet series 
\begin{equation*}
L(s)=\sum_{n=0}^{+\infty} \frac{(-1)^n}{(2n+1)^s}
\end{equation*}
associated with the Dirichlet character of $\mathbb{Z}/4\mathbb{Z}$. More precisely, for every $\mu \geq 0$, the following results holds:
\begin{equation}\label{Mel1}
\mathbb{E}(|A_t|^{\mu})=\frac{4 \Gamma(1+\mu)}{\pi^{1+\mu}} L(1+\mu)t^{\mu},
\end{equation}
where $\mathbb{E}$ stands for the expectation of the underlying probability space in which $B$ is defined. At the heart of the derivation of \eqref{Mel1} is the Dubins-Schwarz Theorem (\cite{Rev-Yor}) together with the knowledge of the Mellin transform:
\begin{equation}\label{Mel2}
\mathbb{E}\left(\int_0^1 R_s^2 ds\right)^{\mu/2}  
\end{equation} 
where $(R_s)_{s \geq 0}$ is a two-dimensional Bessel process (see Table 2 in \cite{BPY}). 

In \cite{Bau-Wan}, generalized L\'evy stochastic areas were defined by replacing the horizontal part of $\eta_H$, that is $\alpha_H$, by those of contact forms of the Anti-de Sitter and the Hopf fibrations. Likewise, these contact forms are the pull-back of  standard 
K\"ahler forms of the hyperbolic ball and the complex projectif space respectively. There, the authors use Girsanov Theorem to derive the characteristic functions of the generalized L\'evy stochastic areas at any fixed time $t > 0$ and prove that they converge in distribution as $t \rightarrow +\infty$ to Gaussian and Cauchy random variables respectively. However, the densities of these random areas were missed and we shall derive them in the first part of this paper. For the Anti de Sitter fibration, our strategy shows a close connection with the so-called generalized Maass Laplacian and has the merit to make transparent the convergence in distribution to the Gaussian random variable. Moreover, the Maass Laplacian in the Poincar\'e disc has its origin in number theory and in this complex one-dimensional setting, we shall further express the Mellin transform of the generalized stochastic area as a series of Riemann zeta functions. Doing so establishes a connection with number-theoretical objects in the same spirit \eqref{Mel1} and \eqref{Mel2} do. As to the Hopf fibration, the derivation of the density of the corresponding generalized stochastic area is rather direct. Indeed, we shall work out the integral representation of its characteristic function proved in \cite{Bau-Wan} using the circular Jacobi semi-group density and obtain a series of Cauchy kernels which we invert termwise. 

The second part of the paper is concerned with the winding processes around the origin in the Poincar\'e disc and in the complex projective line. The generators of these processes are naturally defined as the angular parts of the Brownian motions on the corresponding geometrical models. The characteristic functions of their fixed-time marginals were expressed as integrals of semi-group densities of the hyperbolic Jacobi and the ultraspherical (circular Jacobi with equal parameters) operators. In the hyperbolic setting, the derivation of the windings semi-group density follows readily from the Euler integral representation of the Jacobi function, while the computations relative to windings in the complex projective line are more trickier and more tedious. More precisely, we shall appeal to a suitable representation of even ultraspherical polynomials which stems from a quadratic transformation of the Gauss hypergeometric function and to their Laplace-type integral representation. However, the sought density admits a quite complicated expression and we shall not write it explicitly in order to ease the reading of the paper. 

Before going through computations and for sake of completeness, we collect in the next section the definitions of special functions occurring in the sequel and related results we will use later. The reader is referred to the standard monographs \cite{AAR} and \cite{Gra-Ryz} for a good account. Section 3 and 4 are devoted to the computations related to generalized stochastic areas and section 5 deals with windings processes in the Poincar\'e disc and in the complex projective line.

\section{Special functions}
We start with the Gamma function defined for $x > 0$ by:
\begin{equation*}
\Gamma(x) = \int_0^{\infty} e^{-u}u^{x-1}du.
\end{equation*}
This function satisfies the Legendre duplication formula:
\begin{equation}\label{Leg}
\sqrt{\pi}\Gamma(2x+1) = 2^{2x} \Gamma\left(x+\frac{1}{2}\right)\Gamma(x+1). 
\end{equation}
Next, let $k \geq 1$ be a non negative integer. Then the Pochhammer symbol is defined by: 
\begin{equation*}
(x)_k = (x+k-1)\dots (x+1)x, \quad x \in \mathbb{R}, 
\end{equation*}
with the convention $(x)_0 := 1$. When $x > 0$, we can express it through the Gamma function as: 
\begin{equation*}
(x)_k = \frac{\Gamma(x+k)}{\Gamma(x)}
\end{equation*}
while
\begin{eqnarray}
(-n)_k & = & \frac{(-1)^kn!}{(n-k)!}, \,\, k \leq n, \\ 
 &= & 0, \quad \quad \quad \,\, k > n \nonumber.
\end{eqnarray}
Now, the hypergeometric series ${}_pF_q$ is defined by:
\begin{equation*}
{}_pF_q(a_1, \dots, a_p; b_1, \dots; b_q; x) := \sum_{k \geq 0}\frac{(a_1)_k \dots (a_p)_k}{(b_1)_q\dots(b_q)_k} \frac{x^k}{k!}
\end{equation*}
whenever it converges. In this definition, $(a_i, 1 \leq i \leq p)$ are real numbers while $b_i \in \mathbb{R} \setminus -\mathbb{N}$ for any $1 \leq i \leq q$. In particular, if $a_i = -n$ for some $1 \leq i \leq p$ then the series terminates and we end up with a hypergeometric polynomial. For instance, the Jacobi polynomial of parameters $a, b > -1,$ is represented through the Gauss hypergeometric series:
\begin{equation}\label{RepJac}
P_n^{(a,b)}(x) = \frac{(a+1)_n}{n!}{}_2F_1\left(-n, n+a+b+1, a+1; \frac{1-x}{2}\right).
\end{equation}
These polynomials are orthogonal with respect to the Beta weight $(1-x)^{\alpha}(1+x)^{\beta}$ and their squared $L^2$-norm is given by: 
\begin{equation*}
\int_{-1}^1\left[P_n^{(a,b)}(x)\right]^2(1-x)^{\alpha}(1+x)^{\beta}dx = \frac{2^{a+b+1}}{2n+a+b+1} \frac{\Gamma(n+a+1)\Gamma(n+b+1)}{n!\Gamma(n+a+b+1)}.
\end{equation*}
If $a=b \neq -1/2$, then the Jacobi polynomial reduces (with a different normlization) to Gegenbauer polynomials:
\begin{align}\label{DefGeg}
C_j^{(a+1/2)}(\cos(2r)) & =  \frac{(2a+1)_j}{(a+1)_j}P_j^{(a, a)}(\cos(2r)) \nonumber 
\\& = \frac{(2a+1)_j}{j!}{}_2F_1(-j, j+ 2a+1, a+1, \sin^2(r)). 
\end{align}
These polynomials admit the following Laplace-type integral representation: 
\begin{equation}\label{LapUlt}
\frac{\Gamma(a+1/2)\sqrt{\pi}}{\Gamma(a+1)} \frac{C_{j}^{(a+1/2)}(\cos(2r))}{C_{j}^{(a+1/2)}(1)} = \int_0^{\pi} \left(\cos(2r) + i\sin(2r)\cos\eta\right)^{j} \sin^{2a}(\eta) d\eta. 
\end{equation}
Furthermore, $C_{2j}^{(a+1/2)}$ is an even polynomial and may be expressed as:
\begin{equation}\label{Even}
 C_{2j}^{(a+1/2)}(v) =  (-1)^j\frac{(a+1/2)_j}{j!}{}_2F_1\left(-j, j+a+\frac{1}{2}, \frac{1}{2}, v^2\right).
 \end{equation}
Finally, let $\alpha > -1, \beta \in \mathbb{R}$ and set $\rho := \alpha + \beta + 1$. Then, the Jacobi function of spectral parameter $\mu \in \mathbb{R}$ is defined by (\cite{Koor}): 
\begin{align*}
\phi_{\mu}^{(\alpha, \beta)}(r) & = {}_2F_1\left(\frac{\rho + i\mu}{2}, \frac{\rho - i\mu}{2}, \alpha + 1; -\sinh^2(r)\right), \quad r > 0,
\end{align*}
where ${}_2F_1$ is now the Gauss hypergeometric function which is the analytic extension of the Gauss hypergeometric series to $\mathbb{C} \setminus [1,\infty[$. In particular, we have the Euler integral representation: 
\begin{equation}\label{Euler}
\phi_{\mu}^{(\alpha, \beta)}(r) = \frac{\Gamma(\alpha+1)}{\Gamma((\rho + i\mu)/2)\Gamma(\alpha+1-(\rho+i\mu)/2)} \int_0^1 \frac{u^{(\rho+i\mu)/2-1}(1-u)^{(\alpha-(\rho+i\mu)/2)-1}}{(1+\sinh^2(r)u)^{(\rho -i\mu)/2}} du, 
\end{equation}
provided that the integral converges absolutely. Let us also mention that for special values of their parameters, $\phi_{\mu}^{(\alpha, \beta)}$ and $C_j^{(a+1/2)}$ are zonal spherical functions in hyperbolic spaces and spheres respectively.

\section{The anti-de Sitter fibration}
The Anti de Sitter space is the hypersurface in $\mathbb{C}^{n+1}$ defined by: 
\begin{equation*}
\textrm{AdS}_n := \{(z_1, \dots, z_{n+1}) \in \mathbb{C}^{n+1}, |z_1|^2+|z_2|^2 +\dots + |z_n|^2 - |z_{n+1}|^2 = -1\}.
\end{equation*}
It inherits from $\mathbb{R}^{2n,2}$ a Lorentzian $(2n, 1)$-metric of constant negative curvature and the circle acts by on it in a natural way. The coset space of this action is isometric the complex hyperbolic ball $\mathbb{C}\mathbb{H}^n$ and the projection map 
\begin{equation*}
\textrm{AdS}_n \rightarrow \mathbb{C}\mathbb{H}^n, 
\end{equation*}
is indeed a fibration. In the chart $\{z_{n+1} \neq 0\}$, this fibration sends the coordinate $z_j$ to $w_j = z_j/z_{n+1}$ giving rise to inhomogeneous coordinates $(w_1, \dots, w_n)$ in $\mathbb{C}\mathbb{H}^n$. Let $(w(t))_{t \ge 0}$ be a Brownian motion on 
$\mathbb{CH}^n$ started at the origin $(w_1 = \dots = w_n = 0)$. Then its horizontal lift was computed in \cite{Bau-Wan}, Theorem 3.2, and the argument of the corresponding fiber coordinate is the so-called generalized stochastic area:
\begin{equation*}
\theta_{t,n} := \frac{i}{2} \int_0^t \frac{w(s)  d\overline{w}(s)-\overline{w}(s) dw(s)}{1-\|w(s)\|^2}, \quad \|w(s)\|^2 := \sum_{j=1}^n |w_j(s)|^2,
\end{equation*}
where the above stochastic integral is understood either in the Stratonovich or in the It\^o senses. Recall also from \cite{Bau-Wan} that 
\begin{equation}\label{ADS}
\mathbb{E}\left[e^{i\lambda \theta_{t,n}}\right] = e^{n |\lambda| t} \int_0^{\infty} q_t^{(n-1, |\lambda|)}(0,r) \frac{dr}{(\cosh(r))^{|\lambda|}} 
\end{equation}
where $q_t^{(n-1, |\lambda|)}$ is the heat kernel with respect to Lebesgue measure of the hyperbolic Jacobi generator: 
\begin{equation*}
\mathcal{L}^{(n-1, |\lambda|)} = \frac{1}{2} \left[\partial_r^2 + ((2n-1)\coth r + (2|\lambda| + 1)\tanh(r)) \partial_r\right], \quad r \geq 0,
\end{equation*}
subject to Neumann boundary condition at $r = 0$. The spectral decomposition of this operator is known (see e.g. \cite{Koor} and references therein): its spectrum is purely continuous and is given by the set
\begin{equation*}
\{-(\mu^2 + \rho^2), \mu \in \mathbb{R}, \, \rho := (n-1) + |\lambda| + 1 = n + |\lambda|\},
\end{equation*}
corresponding to the Jacobi function $\phi_{\mu}^{(n-1, |\lambda|)}$. As a matter of fact, the heat kernel admits the following integral representation (\cite{CGM}): 
\begin{equation*}
q_t^{(n-1,|\lambda|)}(0,r) = \frac{1}{\pi} \int_0^{\infty}e^{-(\mu^2+\rho^2)t/2} \phi_{\mu}^{(n-1, |\lambda|)}(r) \frac{d\mu}{|c(\mu)|^2},
\end{equation*}
where $c(\mu)$ is the Harish-chandra function (\cite{Koor}). However, the derivation of the density of $\theta_{t,n}$ we prove below appeals to another representation of $q_t^{(n-1,|\lambda|)}(0,r)$ which involves the heat kernel with respect to the volume measure of the $2n+1$-dimension real hyperbolic space $H^{2n+1}$ (see e.g. \cite{Wan}): 
\begin{align}\label{Hyper}
s_{t,2n+1}(\cosh(x)) & = 
\frac{e^{-n^2t/2}}{(2\pi)^n\sqrt{2\pi t}}\left(-\frac{1}{\sinh(x)}\frac{d}{dx}\right)^ne^{-x^2/(2t)}, \quad x > 0. 
\end{align}
Up to our best knowledge, this new representation have never appeared elsewhere and stems from the intertwining relation between $\mathcal{L}^{(n-1, |\lambda|)}$ and the radial part of the generalized Maass Laplacian below. The issue of our computations is summarized in the following theorem:
\begin{teo}
The density $f_{t,n}$ of $\theta_{t,n}, t > 0,$ is given by:
\begin{multline*}
f_{t,n}(v) = \frac{2e^{-v^2/(2t)}}{\sqrt{2\pi t}} \int_0^{\infty} dr \cosh (r) \sinh^{2n-1}(r) \\  \int_{0}^{\infty} du \cos\left(\frac{uv}{t}\right)e^{u^2/(2t)} s_{t,2n+1}(\cosh(u)\cosh(r)), \,\quad v \in \mathbb{R}.
\end{multline*}
\end{teo}
\begin{proof}
The intertwining relation we alluded to above is: 
\begin{equation}\label{Inter}
 \frac{1}{\cosh^{|\lambda|}(r)} L^{n-1, |\lambda|}(\cdot) (r) = \left[\mathcal{L}^{n-1, |\lambda|}+ \frac{\rho^2}{2}\right] \left(\frac{1}{\cosh^{|\lambda|}} \cdot\right)(r), 
\end{equation}
and follows from straightforward computations. Here,
\begin{equation*}
2L^{(n-1, |\lambda|)} = \partial_r^2 + ((2n-1)\coth r + \tanh(r)) \partial_r + \frac{\lambda^2}{\cosh^2(r)} + n^2
\end{equation*}
is the radial part of the shifted generalized Maass Laplacian (see Proposition 2.1. in \cite{Aya-Int}, see also \cite{Int-OM}):
\begin{equation*}
D^{(n,|\lambda|/2)} + n^2 := 4(1-|w|^2)\left\{ \sum_{i,j=1}^n(\delta_{ij}- w_iw_j) \partial_{w_i\overline{w_j}}^2 +\frac{|\lambda|}{2}\sum_{i=1}^n(w_i\partial_{w_i} - \overline{w_i}\overline{\partial_{w_i}}) + \frac{\lambda^2}{4}\right\} + n^2. 
 \end{equation*}
Consequently, if $v_t^{(n,|\lambda|/2)}(0,y)$ is the heat kernel of $D_{|\lambda|/2, n}$ with respect to the measure (this is the radial part of the volume measure of the complex hyperbolic ball): 
\begin{equation*}
\sinh^{2n-1}(r)\cosh(r)dr,
\end{equation*}
then $v_t^{(n,|\lambda|/2)}(0,y)$ depends only on the hyperbolic distance $r = d(0,y)$, and we get from \eqref{Inter}: 
\begin{equation*}
\frac{e^{n|\lambda|t} e^{\lambda^2t/2}}{(\cosh(r))^{|\lambda|}}  q_t^{(n-1, |\lambda|)}(0,r) = v_t^{(n,|\lambda|/2)}(0,r)\sinh^{2n-1}(r)\cosh(r)dr.
\end{equation*}
Consequently, Theorem 2.2, (i), in \cite{Aya-Int} together with \eqref{Hyper} yield the following expression:
\begin{equation}\label{Expression1}
\mathbb{E}\left[e^{i\lambda \theta_{t,n}}\right]  = 2 e^{-\lambda^2t/2}  \\  \int_0^{\infty} dr \cosh (r) \sinh^{2n-1}(r) \int_{0}^{\infty} dx \sinh(x) N_{|\lambda|/2} (x,y) s_{t,2n+1}(\cosh(x)),
\end{equation}
where we set
\begin{equation*}
N_{|\lambda|/2}(x,0,y) := \frac{1}{\sqrt{\cosh^2(x) - \cosh^2(r)_+}}  {}_2F_1\left(-|\lambda|, |\lambda|, \frac{1}{2}; \frac{\cosh(r) - \cosh(x)}{2\cosh(r)}\right).
\end{equation*}
In particular, we readily deduce that the $\theta_{t,n}/\sqrt{t}$ converges weakly to the Gaussian distribution as $t \rightarrow \infty$. 
Now, perform the variable change $\cosh(x) = \cosh(u) \cosh(r), u > 0,$ for fixed $r$ in the inner integral of the RHS of \eqref{Expression1} and use the identity:
\begin{equation*}
{}_2F_1\left(-|\lambda|, |\lambda|, \frac{1}{2}; \frac{1 - \cosh(u)}{2}\right) = \cosh(\lambda u),
\end{equation*}
to get:
\begin{equation}\label{Integral1}
\mathbb{E}\left[e^{i\lambda \theta_{t,n}}\right]  = 2
 \int_0^{\infty} dr \cosh (r) \sinh^{2n-1}(r)  \int_{0}^{\infty} du e^{-\lambda^2t/2} \cosh(\lambda u)s_{t,2n+1}(\cosh(u)\cosh(r)) . 
\end{equation}
Next, recall the generalized Laplace integral (\cite{Gra-Ryz}): 
\begin{equation*}
e^{-\alpha^2x^2} \cos(wx) = \frac{2e^{-w^2/(4\alpha^2)}}{\alpha \sqrt{\pi}}  \int_0^{\infty} e^{-v^2/\alpha^2} \cosh\left(\frac{wv}{\alpha^2}\right)\cos(2vx)dv
\end{equation*}
 where $w \in \mathbb{C}, x \in \mathbb{R}, \alpha > 0$. Specializing this formula to $w = i u, \alpha^2 = t/2, x = \lambda$ and performing the variable change $v \rightarrow v/2$ there, we get:
 \begin{align*}
e^{-\lambda^2t/2} \cosh(\lambda u) & = \frac{e^{u^2/(2t)}}{\sqrt{2\pi t}}  \int_{\mathbb{R}} e^{-v^2/(2t)} \cos\left(\frac{uv}{t}\right)e^{iv\lambda}dv, 
\end{align*}
whence 
\begin{multline}\label{Laplace}
\mathbb{E}\left[e^{i\lambda \theta_{t,n}}\right]  = 2
 \int_0^{\infty} dr \cosh (r) \sinh^{2n-1}(r) \int_{0}^{\infty} du \frac{e^{u^2/(2t)}}{\sqrt{2\pi t}} s_{t,2n+1}(\cosh(u)\cosh(r))  \\ \int_{\mathbb{R}} e^{-v^2/(2t)} \cos\left(\frac{uv}{t}\right)e^{iv\lambda}dv. 
\end{multline}
Finally, we need to apply Fubini Theorem in \eqref{Laplace} to get the desired density. To this end, recall from \cite{Wan}, eq. 3.25, the estimate
\begin{equation*}
s_{2n+1}(\cosh(\delta)) \leq C \frac{\delta}{\sinh(\delta)} e^{-\delta^2/(2t)}, \quad C, \delta > 0,
\end{equation*}
and note that
\begin{equation*}
\cosh^{-1}[\cosh(u)\cosh(r)] \geq \cosh^{-1}\left[\frac{1}{2}(\cosh(u+r)\right] \geq (r+u), \quad r,u \rightarrow +\infty. 
\end{equation*}
As a matter of fact, $s_{t,2n+1}(\cosh(u)\cosh(r))$ decays as $e^{-(r+u)^2/2t}, r,u \rightarrow +\infty$ so that Fubini Theorem applies. The Theorem is proved.  
\end{proof}

\begin{remark}
The operator $D^{(n, |\lambda|/2)}$ is a deformation of the Laplace-Beltrami operator $D^{(n,0)}$ of $\mathbb{CH}^n$. It is a Laplacian in the sense of Bochner (\cite{Aya-Int}) and allowed in \cite{Bau-Dem} to give another integral representation of the subelliptic heat kernel of the AdS space. For $n=1$, the operator $D^{(1,|\lambda|/2)}$ may be mapped using a weighted Cayley transform to the so-called Maass Laplacian (see Remark 2.1. in \cite{Aya-Int}) in reference to Hans Maass who used it to study weighted automorphic forms (the weight is $|\lambda|/2$). 
\end{remark}

In analogy with the Heisenberg group setting, we shall compute the Mellin transform of $\theta_{t,1}$. In this case, the heat kernel of the three dimensional real hyperbolic space admits a simple expression. Besides, we can express the obtained expression through the Riemann Zeta function.   
\begin{pro}
For integer $m \geq 0$, let 
\begin{equation*}
\mathcal{I}_{m-1/2}(u) = \sum_{j \geq 0}\left(\frac{u}{2}\right)^{2j}\frac{1}{j!\Gamma(m+j+1/2)}, \quad u \in \mathbb{R},
\end{equation*}
be the spherical modified Bessel function and let 
\begin{equation*}
{}_1F_1(a,b,u) = \sum_{j \geq 0} \frac{(a)_j}{(b)_j} \frac{u^j}{j!}, \quad u \in \mathbb{R}, 
\end{equation*}
be the confluent hypergeometric function of the first kind. Then, for any $\mu > 0$, 
\begin{equation*}
\mathbb{E}\left[|\theta_{t,1}|^{\mu-1}\right] = \frac{2(2t)^{(\mu-1)/2}e^{-t/2}}{(2\pi t)^{3/2}}\Gamma\left(\frac{\mu}{2}\right) \int_0^{\infty} \frac{du}{\cosh^2(u)} 
{}_1F_1\left(\frac{\mu}{2}, \frac{1}{2}; -\frac{u^2}{2t}\right)\sum_{m \geq 0}\frac{t^m}{2^m}\mathcal{I}_{m-1/2}(u).
\end{equation*}
In particular, there exists a sequence of real numbers $(c_j(\mu, t))_{j \geq 0}$ such that 
\begin{equation*}
\mathbb{E}\left[|\theta_{t,1}|^{\mu-1}\right] = \frac{2(2t)^{(\mu-1)/2}e^{-t/2}}{(2\pi t)^{3/2}}\Gamma\left(\frac{\mu}{2}\right) \sum_{j \geq 0}\frac{2^{j-1}c_j(\mu, t)}{j!(1-2^{1-j})} \zeta(j),
\end{equation*}
where $\zeta$ stands for the Riemann Zeta function. 
\end{pro}
\begin{proof}
Using the formula (\cite{Gra-Ryz}):
\begin{equation*}
\int_{\mathbb{R}} |v|^{\mu-1} e^{-\beta v^2} \cos(av)dv = \frac{1}{\beta^{\mu/2}}\Gamma\left(\frac{\mu}{2}\right)e^{-a^2/(4\beta)}{}_1F_1\left(\frac{1-\mu}{2}, \frac{1}{2}; \frac{a^2}{4\beta}\right),
\end{equation*}
with $\beta = 1/(2t), a = u/t$, we get:
\begin{multline*}
\mathbb{E}\left[|\theta_{t,n}|^{\mu-1}\right]  = \frac{2(2t)^{(\mu-1)/2}}{\sqrt{\pi}}\Gamma\left(\frac{\mu}{2}\right) 
 \int_0^{\infty} dr \cosh (r) \sinh^{2n-1}(r)  \int_{0}^{\infty} du s_{t,2n+1}(\cosh(u)\cosh(r))  \\  {}_1F_1\left(\frac{1-\mu}{2}, \frac{1}{2}; \frac{u^2}{2t}\right) = \frac{2(2t)^{(\mu-1)/2}}{\sqrt{\pi}}\Gamma\left(\frac{\mu}{2}\right)  
  \int_0^{\infty} du h_{t,2n+1}(u) {}_1F_1\left(\frac{1-\mu}{2}, \frac{1}{2}; \frac{u^2}{2t}\right),
\end{multline*} 
where we set:
\begin{align*}
h_{t,2n+1}(u) & =  \int_0^{\infty} dr \cosh (r) \sinh^{2n-1}(r) s_{t,2n+1}(\cosh(u)\cosh(r)) dr
\\& = \frac{1}{\cosh^2(u)} \int_u^{\infty} \cosh(x)\sinh(x) \left(\frac{\cosh^2(x)}{\cosh^2(u)}-1\right)^{n-1}s_{t,2n+1}(\cosh(x)) dx.
\end{align*}
When $n=1$, we appeal to the formula \eqref{Hyper} to get 
\begin{align*}
\mathbb{E}\left[|\theta_{t,1}|^{\mu-1}\right] &  = \frac{2(2t)^{(\mu-1)/2}e^{-t/2}}{(2\pi t)^{3/2}\sqrt{\pi}}\Gamma\left(\frac{\mu}{2}\right) \int_0^{\infty} \frac{du}{\cosh^2(u)}  {}_1F_1\left(\frac{1-\mu}{2}, \frac{1}{2}; \frac{u^2}{2t}\right)\int_u^{\infty} x\cosh(x)e^{-x^2/(2t)} dx 
\\& =  \frac{(2t)^{(\mu-1)/2}e^{-t/2}}{\sqrt{\pi} (2\pi t)^{3/2}} \Gamma\left(\frac{\mu}{2}\right) \int_0^{\infty} \frac{du}{\cosh^2(u)}  {}_1F_1\left(\frac{1-\mu}{2}, \frac{1}{2}; \frac{u^2}{2t}\right)\sum_{j \geq 0} \int_{u^2}^{\infty} \frac{x^j}{(2j)!} e^{-x/(2t)} dx
\\& =  \frac{(2t)^{(\mu-1)/2}e^{-t/2}}{\sqrt{\pi} (2\pi t)^{3/2}} \Gamma\left(\frac{\mu}{2}\right) \int_0^{\infty} \frac{du e^{-u^2/(2t)}}{\cosh^2(u)}  {}_1F_1\left(\frac{1-\mu}{2}, \frac{1}{2}; \frac{u^2}{2t}\right)\sum_{j \geq 0} \frac{1}{(2j)!}\int_0^{\infty} (x+u^2)^j e^{-x/(2t)} dx
\\& = \frac{(2t)^{(\mu-1)/2}e^{-t/2}}{\sqrt{\pi} (2\pi t)^{3/2}} \Gamma\left(\frac{\mu}{2}\right) \int_0^{\infty} \frac{du e^{-u^2/(2t)}}{\cosh^2(u)} {}_1F_1\left(\frac{1-\mu}{2}, \frac{1}{2}; \frac{u^2}{2t}\right)\sum_{j \geq 0}\frac{j!}{(2j)!}\sum_{m=0}^j \frac{(2t)^{m+1}}{(j-m)!} u^{2(j-m)}.
\end{align*}
Using the Legendre duplication formula: 
\begin{equation*}
\frac{j!}{(2j)!} = \frac{\sqrt{\pi}}{2^{2j}\Gamma(j+1/2)},
\end{equation*}
and changing the order summation in the last series, we further get 
\begin{align*}
\mathbb{E}\left[|\theta_{t,1}|^{\mu-1}\right] & = \frac{2(2t)^{(\mu-1)/2}e^{-t/2}}{(2\pi t)^{3/2}} \Gamma\left(\frac{\mu}{2}\right)\int_0^{\infty} \frac{du e^{-u^2/(2t)}}{\cosh^2(u)} {}_1F_1\left(\frac{1-\mu}{2}, \frac{1}{2}; \frac{u^2}{2t}\right)\sum_{m \geq 0}\frac{t^m}{2^m}
\sum_{j \geq 0}\left(\frac{u}{2}\right)^{2j}\frac{1}{j!\Gamma(m+j+1/2)}
\\& = \frac{2(2t)^{(\mu-1)/2}e^{-t/2}}{(2\pi t)^{3/2}}\Gamma\left(\frac{\mu}{2}\right) \int_0^{\infty} \frac{du e^{-u^2/(2t)}}{\cosh^2(u)} {}_1F_1\left(\frac{1-\mu}{2}, \frac{1}{2}; \frac{u^2}{2t}\right)\sum_{m \geq 0}\frac{t^m}{2^m}\mathcal{I}_{m-1/2}(u). 
\end{align*}
Finally, The first Kummer transformation (\cite{AAR}):
\begin{equation*}
 e^{-u^2/(2t)} {}_1F_1\left(\frac{1-\mu}{2}, \frac{1}{2}; \frac{u^2}{2t}\right) =  {}_1F_1\left(\frac{\mu}{2}, \frac{1}{2}; -\frac{u^2}{2t}\right)
\end{equation*}
yields the first formula of the proposition. As to the second one, it follows from the expansion of the product 
\begin{equation*}
{}_1F_1\left(\frac{\mu}{2}, \frac{1}{2}; -\frac{u^2}{2}\right)\sum_{m \geq 0}\frac{1}{2^m}\mathcal{I}_{m-1/2}(u) := \sum_{j \geq 0} c_j(\mu,t) u^j,
\end{equation*}
and the integral representation:
\begin{equation*}
\zeta(z) = \frac{2^{z-1}}{(1-2^{1-z})\Gamma(z+1)} \int_0^{\infty} \frac{u^z}{\cosh^2(u)} du, \quad \Re(z) > -1. 
\end{equation*}
\end{proof}


As a corollary of the previous proposition, we obtain the Mellin transform of the time change: 
\begin{equation*}
\int_0^t \tanh^2 r(s)ds,
\end{equation*}
where $(r_t)_{t \geq 0}$ is the diffusion with infinitesimal generator $\mathcal{L}^{(0, 0)}$. Indeed, the following equality in distribution was proved in \cite{Bau-Wan}:
\begin{equation}\label{eq-r-theta-H}
\left(\theta_{t,1} \right)_{t \ge 0} \overset{d}{=} \left(\beta_{\int_0^t \tanh^2 r(s)ds} \right)_{t \ge 0},
\end{equation}
where $(\beta_t)_{t \ge 0}$ is a standard Brownian motion independent from $(r_t)_{t \geq 0}$. As a matter of fact, 
\begin{align*}
\mathbb{E}[|\theta_{t,1}|^{\mu-1}]  & = \mathbb{E}[|\beta_1|^{\mu-1}] \mathbb{E}\left(\int_0^t \tanh^2 r(s)ds \right)^{(\mu-1)/2} 
\\& = \frac{2^{1+(\mu/2)}}{\sqrt{2\pi}}\Gamma\left(\frac{\mu}{2}\right) \mathbb{E}\left(\int_0^t \tanh^2 r(s)ds \right)^{(\mu-1)/2},
\end{align*}
whence we readily deduce: 
\begin{cor}
For any $\mu > 0$, 
\begin{equation}\label{Zeta}
\mathbb{E}\left(\int_0^1 \tanh^2 r(s)ds \right)^{(\mu-1)/2} = \frac{1}{\sqrt{2e}(2\pi)} \int_0^{\infty} \frac{du}{\cosh^2(u)} {}_1F_1\left(\frac{\mu}{2}, \frac{1}{2}; -\frac{u^2}{2}\right)\sum_{m \geq 0}\frac{1}{2^m}\mathcal{I}_{m-1/2}(u).
\end{equation}
\end{cor}

\section{The Hopf fibration}
In this section, we deal with the spherical analogue the AdS fibration, commonly known as the Hopf fibration (\cite{Gri}). Here, the base space is the complex projectif space $\mathbb{CP}^n \subset \mathbb{C}^{n+1}$ and the total space is the odd-dimensional sphere
\begin{equation*}
S^{2n+1}:= \{(z_1, \dots, z_{n+1}) \in \mathbb{C}^{n+1}, |z_1|^2+|z_2|^2 +\dots + |z_n|^2 + |z_{n+1}|^2 = 1\},
\end{equation*}
on which the circle acts isometrically. We similarly denote $(w_1, \dots, w_n)$ the inhomogeneous coordinates in the chart $z_{n+1} \neq 0$ and consider a Brownian motion $(w(t))_{t \geq 0}$ on $\mathbb{CP}^n$ starting at zero. Then, the generalized stochastic area process arising from the Hopf fibration is defined by:
\begin{equation*}
\xi_{t,n} := \frac{i}{2} \int_0^t \frac{w(s)  d\overline{w}(s)-\overline{w}(s) dw(s)}{1+\|w(s)\|^2}, \quad \|w(s)\|^2 = \sum_{j=1}^n |w_j(s)|^2.
\end{equation*}
Recall also from \cite{Bau-Wan} that the characteristic function of $\xi_{t,n}$ is given by: 
\begin{equation}\label{Hopf}
\mathbb{E}\left[e^{i\lambda \theta_t}\right] = e^{-n |\lambda| t} \int_0^{\pi/2} p_t^{(n-1, |\lambda|)}(0,r) \frac{dr}{(\cos(r))^{|\lambda|}} 
\end{equation}
where $p_t^{(n-1, |\lambda|)}$ is the heat kernel with respect to Lebesgue measure of the circular Jacobi generator: 
\begin{equation*}
\mathscr{L}^{(n-1, |\lambda|)} = \frac{1}{2} \left[\partial_r^2 + ((2n-1)\cot r - (2|\lambda| + 1)\tan(r)) \partial_r\right], \quad r \in (0,\pi/2),
\end{equation*}
subject to Neumann boundary condition at $r \in \{0, \pi/2\}$. From the the appendix of \cite{Bau-Wan}, we deduce the expansion of $p_t^{(n-1, |\lambda|)}$ in the basis of Jacobi polynomials:
\begin{equation}\label{CirJac} 
p_t^{(n-1, |\lambda|)}(0,r) = \frac{2}{\Gamma(n)}[\cos(r)]^{2|\lambda|+1}[\sin(r)]^{2n-1} \sum_{j \geq 0}(2j+n+|\lambda|) e^{-2j(j+n+|\lambda|)t}\frac{\Gamma(j+n+|\lambda|)}{\Gamma(j+|\lambda|+1)} P_j^{(n-1, |\lambda|)}(\cos(2r)). 
\end{equation}
Rather than performing a kind of a Doob-transform \eqref{Inter} as in the Anti de-Sitter case, we shall derive the density of $\xi_{t,n}$ using direct computations, that is by working out the integral in the RHS of \eqref{Hopf}. This difference between the two methods stems from the fact that the factor $e^{-n|\lambda|t}$ in \eqref{Hopf} encodes the long-time behaviour of $\xi_{t,n}$, while the factor $e^{-\lambda^2t/2}$ (which does the same for $\theta_{t,n}$) is present in \eqref{Expression1} but not in \eqref{ADS}. The issue of our computations is summarized in the following Theorem: 
\begin{teo} 
For any integer $n \geq 1$ and any real $t > 0$, the density $\Phi_{t,n}$ of $\xi_{t,n}$ with respect to Lebesgue measure in $\mathbb{R}$ is given by:
\begin{multline*}
\Phi_{t,n}(v) = \frac{1}{\pi}  \sum_{j \geq 0}(-1)^j \frac{(n)_j}{j!} e^{-2j(j+1)t}  \\ \left\{\frac{2^n(2j+n)t}{[(2j+n)t]^2+v^2} + \int_0^{\infty}e^{-2ju} \left\{\sum_{k=0}^n a_{k,n}(j)e^{-2ku}\right\} \frac{(2j+n)t+u}{[(2j+n)t+u]^2+v^2} du\right\}, \quad v\in \mathbb{R},
\end{multline*}
where for any $j \geq 0$, the coefficients $a_{k,n}(j)$ are defined by the decomposition: 
\begin{align*}
|\lambda|\frac{\Gamma(j+|\lambda|/2)(2j+n+|\lambda|)}{2\Gamma(1+n+j+|\lambda|/2)}\frac{\Gamma(j+n+|\lambda|)}{\Gamma(j+1+|\lambda|)}= 2^n+ \sum_{k=0}^n\frac{a_{k,n}(j)}{2k+2j+|\lambda|}, \quad |\lambda| \geq 0,
\end{align*}
with $a_{0,n}(0) = 0$. 
 \end{teo}
\begin{proof}
Firstly, we plug \eqref{CirJac} in \eqref{Hopf} and use Fubini Theorem to get: 
\begin{multline*}
\mathbb{E}\left[e^{i\lambda \xi_{t,n}}\right] = \frac{2e^{-n |\lambda| t}}{\Gamma(n)}\sum_{j \geq 0}(2j+n+|\lambda|) 
e^{-2j(j+n+|\lambda|)t}\frac{\Gamma(j+n+|\lambda|)}{\Gamma(j+|\lambda|+1)} \\  \int_0^{\pi/2} [\cos(r)]^{|\lambda|+1}[\sin(r)]^{2n-1}P_j^{(n-1, |\lambda|)}(\cos(2r))dr. 
\end{multline*}
Secondly, we need to compute the above integral. To this end, we expand the Jacobi polynomial as
\begin{equation*}
P_j^{(n-1, |\lambda|)}(\cos(2r)) = \frac{(n)_j}{j!} \sum_{m=0}^j \frac{(-j)_m(j+n+|\lambda|)_m}{(n)_mm!}\sin^{2m}(r)
\end{equation*}
and perform the variable change $v= \cos(r)$:
\begin{align*}
 \int_0^{\pi/2} [\cos(r)]^{|\lambda|+1}[\sin(r)]^{2n-1}P_j^{(n-1, |\lambda|)}(\cos(2r))dr & =  \frac{(n)_j}{j!} \sum_{m=0}^j \frac{(-j)_m(j+n+|\lambda|)_m}{(n)_m m!} \int_0^{1} v^{|\lambda|+1}(1-v^2)^{m+n-1}dv
\\& = \frac{\Gamma(n+j)}{2j!} \frac{\Gamma(1+|\lambda|/2)}{\Gamma(1+n+|\lambda|/2)} \sum_{m=0}^j \frac{(-j)_m(j+n+|\lambda|)_m}{(|\lambda|/2 +n+1)_m m!} 
 \\& = \frac{\Gamma(n+j)}{2j!} \frac{\Gamma(1+|\lambda|/2)}{\Gamma(1+n+|\lambda|/2)} {}_2F_1\left(-j, j+n+|\lambda|, 1+n+\frac{|\lambda|}{2}; 1\right).
 \end{align*}
 But, the representation  \eqref{RepJac} of Jacobi polynomials together with the symmetry relation: 
\begin{align*}
P_j^{(a, b)}(v) = (-1)^jP_j^{(b, a)}(-v),
\end{align*}
imply for any $\lambda \neq 0$,  
 \begin{align*}
\int_0^{\pi/2} [\cos(r)]^{|\lambda|+1}[\sin(r)]^{2n-1}P_j^{(n-1, |\lambda|)}(\cos(2r))dr & = \frac{\Gamma(n+j)\Gamma(1+|\lambda|/2)}{2\Gamma(1+n+|\lambda|/2)(1+n+|\lambda|/2)_j} P_j^{(|\lambda|/2+n, |\lambda|/2-1)}(-1)
\\& =  (-1)^j\frac{\Gamma(n+j)\Gamma(1+|\lambda|/2)}{2\Gamma(1+n+j+|\lambda|/2)} P_j^{(|\lambda|/2-1, |\lambda|/2+n)}(1) 
\\& = \frac{(-1)^j\Gamma(n+j)}{2j!}\frac{(|\lambda|/2)_j\Gamma(1+|\lambda|/2)}{\Gamma(1+n+j+|\lambda|/2)}
\\& = \frac{(-1)^j\Gamma(n+j)}{4j!}\frac{\Gamma(j+|\lambda|/2)|\lambda|}{\Gamma(1+n+j+|\lambda|/2)}. 
 \end{align*}
This expression remains valid for $\lambda = 0$ after taking the limit as $\lambda \rightarrow 0$. More precisely, the variable change $v = \cos(2r)$ shows that 
 \begin{equation*}
 \int_0^{\pi/2} \cos(r)[\sin(r)]^{2n-1}P_j^{(n-1, 0)}(\cos(2r))dr = \frac{1}{2^{n+1}}\int_{-1}^1(1-v)^{n-1}P_j^{(n-1, 0)}(v) dv = \frac{1}{2n} \delta_{j0}.
 \end{equation*}
 On the other hand, if $j \neq 0$ then 
 \begin{equation*}
\lim_{\lambda \rightarrow 0} \frac{(-1)^j\Gamma(n+j)}{4j!}\frac{\Gamma(j+|\lambda|/2)|\lambda|}{\Gamma(1+n+j+|\lambda|/2)} = 0,  
 \end{equation*}
 while if $j=0$, then the relation $x\Gamma(x) = \Gamma(x+1)$ yields: 
  \begin{equation*}
\lim_{\lambda \rightarrow 0} \frac{\Gamma(n)}{2}\frac{\Gamma(|\lambda|/2)|\lambda|}{2\Gamma(1+n+|\lambda|/2)} = \frac{1}{2n}.  
 \end{equation*}
As a result, for any $\lambda \in \mathbb{R}$, 
\begin{align*}
\mathbb{E}\left[e^{i\lambda \xi_{t,n}}\right] & = \sum_{j \geq 0}(-1)^j \frac{(n)_j}{j!} e^{-2j(j+1)t} e^{-(2j+n)|\lambda| t}|\lambda|\frac{(2j+n+|\lambda|)\Gamma(j+|\lambda|/2)}{2\Gamma(1+n+j+|\lambda|/2)}\frac{\Gamma(j+n+|\lambda|)}{\Gamma(j+|\lambda|+1)}.
\end{align*}
Finally, it suffices to invert termwise this expansion. To proceed, we write for $j \geq 1$, 
\begin{align*}
|\lambda|\frac{(2j+n+|\lambda|)\Gamma(j+|\lambda|/2)}{2\Gamma(1+n+j+|\lambda|/2)}\frac{\Gamma(j+n+|\lambda|)}{\Gamma(j+1+|\lambda|)} = \frac{|\lambda|(2j+n+|\lambda|)(j+1+|\lambda|)\cdots (j+(n-1)+|\lambda|)}{2(j+|\lambda|/2)\cdots (n+j+|\lambda|/2)}
\end{align*}
as a rational function of $|\lambda|$ which tends to $2^n$ when $|\lambda| \rightarrow \infty$. Consequently, it may be decomposed as 
\begin{align*}
|\lambda|\frac{\Gamma(j+|\lambda|/2)(2j+n+|\lambda|)}{2\Gamma(1+n+j+|\lambda|/2)}\frac{\Gamma(j+n+|\lambda|)}{\Gamma(j+1+|\lambda|)} & = 2^n+ \sum_{k=0}^n \frac{a_{k,n}(j)}{2k+2j+|\lambda|}.
\\& = 2^n + \int_0^{\infty} e^{-|\lambda|u} \sum_{k=0}^na_{k,n}(j)e^{-(2k+2j)u}du.
\end{align*}
Similarly, if $j=0$ then 
\begin{align*}
|\lambda|\frac{(n+|\lambda|)\Gamma(|\lambda|/2)}{2\Gamma(1+n+|\lambda|/2)}\frac{\Gamma(n+|\lambda|)}{\Gamma(1+|\lambda|)} & = \frac{\Gamma(1+|\lambda|/2)}{\Gamma(1+n+|\lambda|/2)}\frac{\Gamma(1+n+|\lambda|)}{\Gamma(1+|\lambda|)} 
\\& = 2^n + \sum_{k=1}^n \frac{a_{k,n}(0)}{2k+|\lambda|}.
\end{align*}
We also notice that the coefficients $a_{k,n}(j), k \in \{0, \dots, n\}$ are polynomial functions of $j$ with uniformly bounded degrees by $n$. As a matter of fact, the formula 
\begin{equation*}
e^{-\gamma |\lambda|} = \int_{\mathbb{R}} e^{i\lambda v} \frac{\gamma}{\gamma^2+v^2} \frac{dv}{\pi}, \quad \gamma > 0, 
\end{equation*}
together with Fubini Theorem lead to the sought density.  
\end{proof}
\begin{remark}
For any $t > 0$, the density of $\xi_{t,n}/t$ is $t\Phi_{t,n}(tv)$ and its limiting behavior as $t \rightarrow \infty$ is given by the term $j=0$ in the above series. From the identity 
\begin{equation*}
\sum_{k=1}^n\frac{a_{k,n}(0)}{2k} = 1- 2^n, 
\end{equation*}
and the dominated convergence Theorem, we readily compute
\begin{equation*}
\lim_{t \rightarrow \infty}t\Phi_{t,n}(tv) = \frac{1}{\pi}\frac{n}{n^2+v^2}. 
\end{equation*}
Consequently, $\xi_{t,n}/t$ converges in distribution as $t \rightarrow \infty$ to a Cauchy random variable of parameter $n$. This limiting result was already proved in \cite{Bau-Wan} directly from the characteristic function \eqref{Hopf}. 
\end{remark}

In the particular case $n=1$ which corresponds to the Riemann sphere, we can easily compute the coefficients $a_{0,1}(j) = -2j$ and $a_{1,1}(j) = -2(j+1)$ for $j \geq 0$ and obtain the following: 
\begin{cor}
The density $\Phi_{t,1}$ of $\xi_{t,1}$ reduces to: 
\begin{multline*}
\Phi_{t,1}(v) = \frac{2}{\pi}  \sum_{j \geq 0}(-1)^j e^{-2j(j+1)t}  \left\{\frac{(2j+1)t}{[(2j+1)t]^2+v^2} - \int_0^{\infty}e^{-2ju} \left(j + (j+1)e^{-2u}\right) \frac{(2j+1)t+u}{[(2j+1)t+u]^2+v^2} du\right\}.
\end{multline*}
\end{cor}

\section{Winding processes in $\mathbb{C}\mathbb{H}^1$ and in $\mathbb{C}\mathbb{P}^1$}
In this section, we are interested in winding processes around the origin in the Poincar\'e disc $\mathbb{C}\mathbb{H}^1$ and in the complex projective line $\mathbb{C}\mathbb{P}^1$. As in the Euclidean setting, these processes are naturally defined as angular parts of Laplace operators of their corresponding geometrical models. In \cite{Bau-Wan}, the characteristic functions of their fixed-time marginal distributions were expressed as expectations with respect to the hyperbolic Jacobi and ultraspherical operators (see below) and their large-time behavior were determined. In the next paragraph, we shall derive the semi-group density of the winding process in $\mathbb{C}\mathbb{H}^1$.     
\subsection{The Poincar\'e disc}
The Laplace operator of the Poincar\'e disc $\mathbb{C}\mathbb{H}^1$ is given by: 
\begin{equation*}
\frac{(1-|w|^2)^2}{2}\partial_{w\overline{w}}, \quad |w| < 1. 
\end{equation*}
This is the infinitesimal generator of the Brownian motion in $\mathbb{C}\mathbb{H}^1$ and reads in cylindrical coordinates $z=\tanh(r)e^{i\phi}, r \geq 0,$: 
\begin{equation*}
\frac{1}{2}\left(\partial_r^2 + 2\coth(2r)\partial_r + \frac{4}{\sinh^2(2r)} \partial_{\phi}^2\right),
\end{equation*}
Accordingly, the winding process is defined by: 
\begin{equation*} 
\phi_t := \beta_{\int_0^t [4/\sinh^2(2r_s)] dr_s}, \quad t \geq 0, 
\end{equation*}
where $(\beta_t)_{t \geq 0}$ is a real Brownian motion independent from $(r_t)_{t \geq 0}$. By rotation invariance of the Laplace operator, we may assume without loss of generality that the Brownian motion in  $\mathbb{C}\mathbb{H}^1$ starts at $\tanh(r_0) \in (0,1)$. In this respect, it was proved in \cite{Bau-Wan} (see the proof of Theorem 4.2) that for any $\lambda \in \mathbb{R}$: 
\begin{equation*}
\mathbb{E}[e^{i\lambda \phi_t}] = \tanh^{|\lambda|}(r_0) \int_0^{\infty} q_t^{(|\lambda|, -|\lambda|)}(r_0,r) \frac{dr}{\tanh^{|\lambda|}(r)}.
\end{equation*}
Here, $q_t^{(|\lambda|, -|\lambda|)}(r_0,r)$ is the heat kernel with respect to Lebesgue measure of the hyperbolic Jacobi operator:
\begin{equation*}
G := \frac{1}{2}\left(\partial_r^2 + [(2|\lambda|+1)\coth(r) + (1-2|\lambda|)\tanh(r)] \partial_r\right),
\end{equation*}
starting at $r_0 > 0$ and acting on smooth functions on $(0,\infty)$ with Neumann boundary condition at $r=0$. Actually, this kernel is given by the following integral representation (\cite{CGM}): 
\begin{equation*}
q_t^{(|\lambda|, -|\lambda|)}(r_0,r) = \frac{2}{\pi} [\tanh(r)]^{2|\lambda|} \sinh(2r)\int_0^{\infty} e^{-(1+\mu^2)t/2} \phi_{\mu}^{(|\lambda|, -|\lambda|)}(r_0)\overline{\phi_{\mu}^{(|\lambda|, -|\lambda|)}(r)}\frac{d\mu}{|c(\mu)|^2},
\end{equation*}
where 
\begin{align*}
\phi_{\mu}^{(|\lambda|, -|\lambda|)}(r) & = {}_2F_1\left(\frac{1 + i\mu}{2}, \frac{1 - i\mu}{2}, |\lambda|+1; -\sinh^2(r)\right), \quad \overline{\phi_{\mu}^{(|\lambda|, -|\lambda|)}} = \phi_{-\mu}^{(|\lambda|, -|\lambda|)},
\end{align*}
and 
\begin{align}\label{HC}
c(\mu) & = 2^{1-i\mu} \frac{\Gamma(|\lambda|+1)\Gamma(i\mu)}{\Gamma[(i\mu+1)/2]\Gamma[|\lambda|+(1+i\mu)/2]} =  \frac{\Gamma(|\lambda|+1)\Gamma(i\mu/2)}{\Gamma[|\lambda|+(1+i\mu)/2]},
\end{align}
is the Harish-Chandra function (we used the Legendre duplication formula to derive the second equality). As a result, 
\begin{multline*}
\mathbb{E}_{r_0}[e^{i\lambda \phi_t}] =\frac{2}{\pi}  [\tanh(r_0)\tanh(r)]^{|\lambda|} \int_0^{\infty}\sinh(2r) dr \int_0^{\infty} e^{-(1+\mu^2)t/2} \phi_{\mu}^{(|\lambda|, -|\lambda|)}(r_0)\phi_{-\mu}^{(|\lambda|, -|\lambda|)}(r) 
\\ \frac{\Gamma[|\lambda|+(1+i\mu)/2]\Gamma[|\lambda|+(1-i\mu)/2]}{\Gamma^2(|\lambda|+1)} \frac{d\mu}{|\Gamma(i\mu/2)|^2}. 
\end{multline*}
The inversion of this Fourier transform is given in the following theorem: 
\begin{teo}
The density of $\phi_t$ is given by: 
\begin{multline*}
H_t(y) := \frac{1}{2\pi^3} \int_0^{\infty} \sinh(2r) dr \int_0^{\infty} \frac{d\mu}{|\Gamma(i\mu)|^2} e^{-(1+\mu^2)t/2} \int_0^1\int_0^1\frac{v^{-(1+i\mu)/2}(1-v)^{-(1-i\mu)/2}dv}{(1+\sinh^2(r)v)^{(1+i\mu)/2}} 
\\ \frac{u^{(i\mu-1)/2} (1-u)^{-(1+i\mu)/2}du}{(1+\sinh^2(r_0)u)^{(1-i\mu)/2}}  \frac{-\ln [\tanh(r_0)\tanh(r)(1-u)(1-v)]}{\ln^2 [\tanh(r_0)\tanh(r)(1-u)(1-v)] + y^2}. 
\end{multline*}
\end{teo}
\begin{proof} 
Using the Euler integral representation \eqref{Euler}, we write:
\begin{multline*}
\phi_{\mu}^{(|\lambda|, -|\lambda|)}(r)  = \frac{\Gamma(|\lambda|+1)}{\Gamma[(1-i\mu)/2]\Gamma[|\lambda|+(1+i\mu)/2]}\\
\int_0^1v^{-(1+i\mu)/2}(1-v)^{|\lambda|-(1-i\mu)/2} \frac{dv}{(1+\sinh^2(r)v)^{(1+i\mu)/2}},
 \end{multline*}
and similarly,
\begin{multline*}
\phi_{-\mu}^{(|\lambda|, -|\lambda|)}(r)  = \frac{\Gamma(|\lambda|+1)}{\Gamma[(1+i\mu)/2]\Gamma[|\lambda|+(1-i\mu)/2]}\\
\int_0^1u^{(i\mu-1)/2}(1-u)^{|\lambda|-(1+i\mu)/2} \frac{du}{(1+\sinh^2(r_0)u)^{(1-i\mu)/2}}. 
 \end{multline*}
Combining these integral representations and keeping in mind \eqref{HC}, we get:
\begin{multline*}
\mathbb{E}_{r_0}[e^{i\lambda \phi_t}] =\frac{2}{\pi}  \int_0^{\infty} \sinh(2r) dr \int_0^{\infty} \frac{d\mu}{|\Gamma[(1+i\mu)/2]\Gamma(i\mu/2)|^2} e^{-(1+\mu^2)t/2} \\ 
\int_0^1\int_0^1\frac{v^{-(1+i\mu)/2}(1-v)^{-(1-i\mu)/2}dv}{(1+\sinh^2(r)v)^{(1+i\mu)/2}} \frac{u^{(i\mu-1)/2} (1-u)^{-(1+i\mu)/2}du}{(1+\sinh^2(r_0)u)^{(1-i\mu)/2}} [\tanh(r_0)\tanh(r)(1-u)(1-v)]^{|\lambda|}.
\end{multline*}
Appealing again to Legendre duplication formula:
\begin{equation*}
\sqrt{\pi} \Gamma(i\mu) = 2^{i\mu-1}\Gamma\left(i\frac{\mu}{2}\right)\Gamma\left(\frac{1+i\mu}{2}\right),
\end{equation*}
we further get: 
\begin{multline*}
\mathbb{E}_{r_0}[e^{i\lambda \phi_t}] =\frac{1}{2\pi^{2}} \int_0^{\infty} \sinh(2r) dr \int_0^{\infty}  \frac{d\mu}{|\Gamma(i\mu)|^2} e^{-(1+\mu^2)t/2} \int_0^1\int_0^1\frac{v^{-(1+i\mu)/2}(1-v)^{-(1-i\mu)/2}dv}{(1+\sinh^2(r)v)^{(1+i\mu)/2}} 
\\ \frac{u^{(i\mu-1)/2} (1-u)^{-(1+i\mu)/2}du}{(1+\sinh^2(r_0)u)^{(1-i\mu)/2}} [\tanh(r_0)\tanh(r)(1-u)(1-v)]^{|\lambda|}.
\end{multline*}
Since 
\begin{align*}
[\tanh(r_0)\tanh(r)(1-u)(1-v)]^{|\lambda|} & = e^{|\lambda| \ln [\tanh(r_0)\tanh(r)(1-u)(1-v)]} 
\\& = \frac{1}{\pi}\int_{\mathbb{R}} e^{i\lambda y}\frac{-\ln [\tanh(r_0)\tanh(r)(1-u)(1-v)]}{\ln^2 [\tanh(r_0)\tanh(r)(1-u)(1-v)] + y^2} dy,
\end{align*}
and since the multiple integral converges absolutely, then Fubini Theorem applies and yields the desired density. 
\end{proof}

\subsection{The complex projective line}
Let $\mathbb{C}\mathbb{P}^1 = \mathbb{C} \cup \infty$ denote the complex projective line. Then, the generator of the Brownian motion in $\mathbb{C}\mathbb{P}^1$ reads in cylindrical coordinates $z = \tan(r)e^{i\theta}, r \in (0,\pi/2),$: 
\begin{equation*}
\frac{1}{2}\left(\partial_r^2 + 2\cot(2r)\partial_r + \frac{4}{\sin^2(2r)} \partial_{\phi}^2\right). 
\end{equation*}
Similarly, its winding process is defined as the time-changed real Brownian motion: 
\begin{equation*} 
\psi_t := \beta_{\int_0^t [4/\sin^2(2r_s)] dr_s}, \quad t \geq 0,
\end{equation*}
where $(r_s)_{s \geq 0}$ is now the Jacobi diffusion starting at $r_0 \in (0,\pi/2)$ whose generator is given by: 
\begin{equation*}
\frac{1}{2}\left(\partial_r^2 + 2\cot(2r)\partial_r\right),
\end{equation*}
and $(\beta_t)_{t \geq 0}$ is a real Brownian motion independent from $(r_t)_{t \geq 0}$. In \cite{Bau-Wan}, the authors proved that for any $\lambda \in \mathbb{R}$ (see the proof of Theorem 4.1)\footnote{As in the hyperbolic setting, the rotation invariance of the Laplace operator allows to assume without loss of generality that the Brownian motion in $\mathbb{CP}^1$ starts at $r_0 \in (0,\pi/2)$.}: 
\begin{equation}\label{CFW}
\mathbb{E}[e^{i\lambda \phi_t}] = \sin^{|\lambda|}(2r_0)e^{-2(\lambda^2+|\lambda|)t} \int_0^{\pi/2} p_t^{(|\lambda|, |\lambda|)}(r_0,r) \frac{dr}{\sin^{|\lambda|}(2r)}, 
\end{equation}
where $p_t^{(|\lambda|, |\lambda|)}(r_0,r)$ is the semi-group density with respect to Lebesgue measure of the ultraspherical operator: 
\begin{equation*}
\frac{1}{2}\partial_r^2 + \left(\left(|\lambda|+\frac{1}{2}\right) \cot(r) - \left(|\lambda|+\frac{1}{2}\right)\tan(r)\right)\partial_r, 
\end{equation*}
in $(0,\pi/2)$ with Neumann boundary conditions. In the following proposition, we compute the integral displayed in the RHS of \eqref{CFW} as a series of ultraspherical polynomials in the variable $r_0$: 

\begin{pro}\label{proCP1}
For any $t > 0$, 
\begin{multline*}
\mathbb{E}[e^{i\lambda \phi_t}] = \frac{\Gamma(|\lambda|+1/2)\sqrt{\pi}}{4\Gamma(|\lambda|+1)} \sin^{|\lambda|}(2r_0)e^{-2(\lambda^2+|\lambda|)t}\sum_{j \geq 0} e^{-4j(2j+2|\lambda|+1)t} \left(4j+2|\lambda|+1\right)
\\ \frac{|\lambda|\Gamma(j+|\lambda|/2)}{\Gamma(|\lambda|/2+j+3/2)}\frac{\Gamma(|\lambda|+j+1/2)}{\Gamma(|\lambda|+1/2)j!} \frac{C_{2j}^{(|\lambda|+1/2)}(\cos(2r_0))}{C_{2j}^{(|\lambda|+1/2)}(1)}. 
\end{multline*}
\end{pro}
\begin{proof}
The ultraspherical operator is a circular Jacobi operator with equal parameters. From the appendix of \cite{Bau-Wan}, the semi-group density $p_t^{(|\lambda|, |\lambda|)}$ admits the following expansion:
\begin{multline*}
p_t^{(|\lambda|, |\lambda|)}(r_0,r)  = \frac{\sin^{2|\lambda|+1}(2r)}{2^{2|\lambda|}}\sum_{j \geq 0}(2j+2|\lambda|+1)\frac{j!\Gamma(j+2|\lambda|+1)}{\Gamma^2(j+|\lambda|+1)} \\ 
e^{-2j(j+2|\lambda|+1)t}P_j^{(|\lambda|, |\lambda|)}(\cos(2r_0))P_j^{(|\lambda|, |\lambda|)}(\cos(2r)).
\end{multline*} 
From the relation between Jacobi and ultraspherical polynomials and using Legendre duplication formula, we can rewrite this kernel as: 
\begin{multline}\label{SGU}
p_t^{(|\lambda|, |\lambda|)}(r_0,r) = \frac{\Gamma(|\lambda|+1/2)}{\Gamma(|\lambda|+1)\sqrt{\pi}} [\sin(2r)]^{2|\lambda|+1} \sum_{j \geq 0}(2j+2|\lambda|+1) e^{-2j(j+2|\lambda|+1)t} 
\\ \frac{C_j^{(|\lambda|+1/2)}(\cos(2r_0))C_j^{(|\lambda|+1/2)}(\cos(2r))}{C_j^{(|\lambda|+1/2)}(1)}. 
 \end{multline}
By the virtue of \eqref{CFW} and of the bound: 
\begin{equation*}
|C_j^{(|\lambda|+1/2)}(\cos(2r))| \leq C_j^{(|\lambda|+1/2)}(1) = \frac{(2|\lambda|+1)_j}{j!},
\end{equation*}
 we are lead to the following integral: 
\begin{align*}
\int_0^{\pi/2} C_j^{(|\lambda|+1/2)}(\cos(2r))[\sin(2r)]^{|\lambda|+1} dr & = \frac{1}{2}\int_{-1}^{1}C_j^{(|\lambda|+1/2)}(v) (1-v^2)^{|\lambda|/2} dv.
 \end{align*}
 This integral vanishes when $j$ is odd since $C_j^{|\lambda|+1/2}$ is an odd polynomial. Otherwise, we shall appeal to \eqref{Even} to write: 
\begin{align*}
\frac{1}{2}\int_{-1}^{1} C_{2j}^{(|\lambda|+1/2)}(v) (1-v^2)^{|\lambda|/2} dv & = \frac{(-1)^j}{2}\frac{(|\lambda|+1/2)_j}{j!} \sum_{m=0}^j\frac{(-j)_m(j+|\lambda|+1/2)_m}{m!(1/2)_m} 
\int_0^1v^{m-1/2}(1-v)^{|\lambda|/2} dv  
\\& = \frac{(-1)^j}{2}\frac{\sqrt{\pi}\Gamma(|\lambda|/2+1)}{\Gamma((|\lambda|+3)/2)}\frac{(|\lambda|+1/2)_j}{j!}  {}_2F_1\left(-j, j +|\lambda|+\frac{1}{2}, \frac{|\lambda|+3}{2}; 1\right). 
 \end{align*}
 But, if $\lambda \neq 0$ then,
 \begin{align*}
  {}_2F_1\left(-j, j +|\lambda|+\frac{1}{2}, \frac{|\lambda|+3}{2}; 1\right) & = \frac{j!}{((|\lambda|+3)/2)_j}P_j^{((|\lambda|+1)/2, |\lambda|/2 - 1)}(-1) 
  \\& = (-1)^j  \frac{j!}{((|\lambda|+3)/2)_j} P_j^{(|\lambda|/2 -1, (|\lambda|+1)/2)}(1)
 \\& = (-1)^j  \frac{(|\lambda|/2)_j}{((|\lambda|+3)/2)_j},
  \end{align*}
whence,
 \begin{align*}
\frac{1}{2}\int_{-1}^{1} C_{2j}^{(|\lambda|+1/2)}(v)(1-v^2)^{|\lambda|/2} dv 
& = \frac{1}{2}\frac{(|\lambda|+1/2)_j}{j!}\frac{\sqrt{\pi}\Gamma(|\lambda|/2+1)}{\Gamma((|\lambda|+3)/2)} \frac{(|\lambda|/2)_j}{((|\lambda|+3)/2)_j}
\\& = \frac{\sqrt{\pi}|\lambda|\Gamma(j+|\lambda|/2)}{4\Gamma(|\lambda|/2+j+3/2)}\frac{\Gamma(|\lambda|+j+1/2)}{\Gamma(|\lambda|+1/2)j!}.
\end{align*}
This expression remains valid when $|\lambda| = 0$ if we take the limit $\lambda \rightarrow 0$. Indeed, 
\begin{equation*}
C_{j}^{(1/2)} = P_{j}^{(0,0)}, \quad j \geq 0,
\end{equation*}
are Legendre polynomials which are orthogonal with respect to the uniform measure in $[-1,1]$. As a matter of fact, 
\begin{equation*}
\frac{1}{2}\int_{-1}^{1} C_{2j}^{(1/2)}(v) dv = \delta_{j0}. 
\end{equation*}
On the other hand, if $j \neq 0$ then 
\begin{equation*}
\lim_{\lambda \rightarrow 0} \frac{\sqrt{\pi}|\lambda|\Gamma(j+|\lambda|/2)}{4\Gamma(|\lambda|/2+j+3/2)}\frac{\Gamma(|\lambda|+j+1/2)}{\Gamma(|\lambda|+1/2)j!} = 0,
\end{equation*}
while when $j=0$, the relation $x\Gamma(x) = \Gamma(x+1)$ yields:
\begin{equation*}
\lim_{\lambda \rightarrow 0} \frac{\sqrt{\pi}|\lambda|\Gamma(|\lambda|/2)}{4\Gamma(|\lambda|/2+3/2)} = \lim_{\lambda \rightarrow 0} \frac{\sqrt{\pi}\Gamma(1+ |\lambda|/2)}{2\Gamma(|\lambda|/2+3/2)} = 1. 
\end{equation*}
Remembering \eqref{SGU}, we are done. 
\end{proof} 
Now, \eqref{LapUlt} allows to express
\begin{align}\label{Four1}
\frac{\Gamma(|\lambda|+1/2)\sqrt{\pi}}{\Gamma(|\lambda|+1)}\frac{C_{2j}^{(|\lambda|+1/2)}(\cos(2r_0))}{C_{2j}^{(|\lambda|+1/2)}(1)} & = \int_0^{\pi} \left(\cos(2r_0) + i\sin(2r_0)\cos\eta\right)^{2j} \sin^{2|\lambda+1}(\eta) d\eta \nonumber 
\\&= \int_{\mathbb{R}} e^{i\lambda v} \int_0^{\pi} \left(\cos(2r_0) + i\sin(2r_0)\cos\eta\right)^{2j} \sin(\eta)\frac{-2\ln(\sin(\eta))}{4\ln^2(\sin(\eta))+v^2}  d\eta \frac{dv}{\pi},
\end{align}
while for any $j \geq 0$, 
\begin{align}\label{Four2}
\frac{|\lambda|/2\Gamma(j+|\lambda|/2)}{\Gamma(|\lambda|/2+j+3/2)} &= \frac{1}{\Gamma(1/2)}\int_0^1 \eta^{|\lambda|/2+j}\frac{d\eta}{\sqrt{1-\eta}} \nonumber
\\& =  \frac{2}{\sqrt{\pi}} \int_{\mathbb{R}} e^{i\lambda v} \int_0^1 \eta^{j}\frac{-\ln(\eta)}{\ln^2(\eta)+4v^2}  \frac{d\eta}{\sqrt{1-\eta}} \frac{dv}{\pi},
\end{align}
and 
\begin{eqnarray}\label{Four3}
\sin^{|\lambda|}(2r_0) & = & \int_{\mathbb{R}} e^{i\lambda v}\frac{-\ln(\sin(2r_0)))}{\ln^2(\sin(\eta))+v^2}  \frac{dv}{\pi}, \quad r _0 \in (0,\pi/2), \\
e^{-2\lambda^2t} & =&   \int_{\mathbb{R}} e^{i\lambda v} e^{-v^2/(8t)} \frac{dv}{2\sqrt{2\pi}}.
\end{eqnarray}
Finally, for any $j \geq 0$, 
\begin{equation*}
\left(4j+2|\lambda|+1\right)\frac{\Gamma(|\lambda|+j+1/2)}{\Gamma(|\lambda|+1/2)} e^{-2|\lambda| t(4j+1)} =  Y_j(|\lambda|)e^{-2|\lambda| t(4j+1)}
\end{equation*}
for some polynomial $Y$ in $|\lambda|$ of degree $j+1$. As a matter of fact, 
 \begin{equation*}
\left(4j+2|\lambda|+1\right)\frac{\Gamma(|\lambda|+j+1/2)}{\Gamma(|\lambda|+1/2)} e^{-2|\lambda| t(4j+1)} =  (-1)^{j+1}Y_j\left(\frac{d}{ds}\right)e^{-|\lambda| s} {}_{|s=(8j+2)t},
\end{equation*}
or equivalently 
\begin{equation}\label{Four4}
\left(4j+2|\lambda|+1\right)\frac{\Gamma(|\lambda|+j+1/2)}{\Gamma(|\lambda|+1/2)} e^{-2|\lambda| t(4j+1)} =  (-1)^{j+1} \int_{\mathbb{R}} e^{i\lambda v} Y_j\left(\frac{d}{ds}\right)\frac{s}{s^2+v^2} {}_{|s=(8j+2)t} \frac{dv}{\pi}.
\end{equation}
Combining \eqref{Four1}, \eqref{Four2}, \eqref{Four3} and \eqref{Four4}, we can write for any integer $j \geq 0$
\begin{multline*}
\frac{|\lambda|/2\Gamma(j+|\lambda|/2)}{\Gamma(|\lambda|/2+j+3/2)}  \frac{\Gamma(|\lambda|+1/2)\sqrt{\pi}}{\Gamma(|\lambda|+1)}\frac{C_{2j}^{(|\lambda|+1/2)}(\cos(2r_0))}{C_{2j}^{(|\lambda|+1/2)}(1)} \sin^{|\lambda|}(2r_0) e^{-2\lambda^2t}
\\ \left(4j+2|\lambda|+1\right)\frac{\Gamma(|\lambda|+j+1/2)}{\Gamma(|\lambda|+1/2)} e^{-2|\lambda| t(4j+1)} 
\end{multline*}
as a Fourier transform in the variable $\lambda$ of some function $k_{t,j}(v)$.  Keeping in mind Proposition \ref{proCP1}, we deduce: 
\begin{cor}
For any $t \geq 0$, the density of $\phi_t$ is given by 
\begin{equation*}
K_t(v) = \frac{1}{2} \sum_{j \geq 0}e^{-4j(2j+1)t} k_{t,j}(v), \quad v \in \mathbb{R}.
\end{equation*}
\end{cor}


\begin{thebibliography}{99}
\bibitem{AAR}\emph{G. E. Andrews, R. Askey, R. Roy}. Special functions. {\it Cambridge University Press}. 1999.
\bibitem{Aya-Int}\emph{K. Ayaz, A. Intissar}. Selberg trace formulae for heat and wave kernels of Maass Laplacians on compact forms of the complex hyperbolic space $H_n(\mathbb{C}), n \geq 2$. {\it Diff. Geom. Appl}. 2001. 
\bibitem{Bau-Dem}\emph{F. Baudoin, N. Demni}. Integral representation of the sub-elliptic heat kernel on the Anti-de Sitter space. {\it Archiv der Math} (Basel). {\bf 111}, (2018), no. 4. 399-406.
\bibitem{Bau-Wan}\emph{F. Baudoin, J. Wang} Stochastic areas, winding numbers and Hopf fibrations. {\it Probab. Theory Related Fields}, {\bf 169}, (2017), no. 3-4, 977-1005.
\bibitem{BPY}\emph{P. Biane, J. Pitman, M. Yor}. Probability laws related to the Jacobi theta and Riemann zeta functions, and Brownian excursions. {\it Bull. Amer. Math. Soc. (N.S.)} {\bf 38}, (2001), no. 4, 435-465. 
\bibitem{CGM}\emph{F, Chouchene, L. Gallardo, M. Mili}. The Heat Semigroup for the Jacobi?Dunkl Operator and the Related Markov Processes. {\it Potential Analysis.} {\bf Vol. 25}, (2006), Issue 2, 103-119. 
\bibitem{Gav}\emph{B. Gaveau}. Principe de moindre action, propagation de la chaleur et estim\'ees sous elliptiques sur certains groupes nilpotents. {\it Acta Math}. {\bf 139}, (1-2), 95-153, (1977).
\bibitem{Gra-Ryz} \emph{I.S. Gradshteyn, I.M. Ryzhik}, Table of integrals, series and products, {\it 5th ed., Academic Press, Boston, MA}, 1994.
\bibitem{Gri}\emph{E. L. Grinberg}. Spherical harmonics and integral geometry on projective spaces. {\it Trans. Amer. Math. Soc.}, {\bf 279}, no. 1. 1983. 
\bibitem{Int-OM}\emph{A. Intissar, M. V. Ould Moustapha}. Explicit formulae for the wave kernels for the laplacians $\Delta_{\alpha \beta}$ in the Bergman Ball $B^n, n \geq 1$. {\it Ann. Glo. Anal. Geom}. {\bf 15}, (1997), 221-234. 
\bibitem{Koor}\emph{T. Koornwinder}. Jacobi functions and analysis on non compact semi simple Lie groups. {\it Special functions: group theoretical aspects and applications}, 1-85, {\it Math. Appl., Reidel, Dordrecht}, 1984. 
\bibitem{Mo}\emph{Z. Mouayn}. Coherent states attached to Landau levels on the Poincar\'{e} disc, \textit{J. Phys. A: Math. Gen}. {\bf 38}, (2005) 9309-9316. 
\bibitem{Rev-Yor}\emph{D. Revuz, M. Yor}. Continuous Martingales And Brownian Motion, $3^{\textrm{rd}}$ ed, Springer, 1999.
\bibitem{Wan}\emph{J. Wang}. The Subelliptic Heat Kernel on the anti-de Sitter spaces. {\it Journal of Potential Analysis}. (2016), {\bf 45}, (4), 635-653. 
\end{thebibliography}
\end{document}